\documentclass[12pt]{amsart}
\usepackage[margin=1in]{geometry}

\usepackage{amsthm,amsmath,amssymb,color,setspace}
\usepackage{mathtools,leftidx,tensor}
\usepackage[hidelinks]{hyperref}
\usepackage{graphicx}
\usepackage{xurl}

\usepackage[foot]{amsaddr}

\usepackage{lineno}

\theoremstyle{definition}
\newtheorem{defi}{Definition}[section]

\theoremstyle{plain}
\newtheorem{thm}[defi]{Theorem}

\newtheorem{rem}{Remark}
\newcommand{\lr}[1]{\lbrace #1 \rbrace}

\makeatletter
\def\imod#1{\allowbreak\mkern10mu({\operator@font mod}\,\,#1)} 
\def\@setcopyright{}                                           
\def\serieslogo@{}
\makeatother

\begin{document}
\singlespacing

\author[D.V.A.~Briones]{Dom Vito A. ~Briones}
\address[D.V.A.~Briones, Corresponding author]{Institute of Mathematics, University of the Philippines Diliman, 1101 Quezon City, Philippines}
\email{dabriones@up.edu.ph}

\title{Association schemes on triples from affine special semilinear groups}

\begin{abstract} 
Association schemes on triples (ASTs) are 3-dimensional analogues of classical association schemes. If a group acts two-transitively on a set, the orbits of the action induced on the triple Cartesian product of that set yields an AST. By considering the actions of semidirect products of the affine special linear group $ASL(k,n)$ with subgroups of the Galois group $Gal(GF(n))$, we obtain the sizes, third valencies, and intersection numbers of the ASTs obtained from subgroups of the affine special semilinear group. 
\end{abstract}



\date{\today}

\maketitle
\singlespacing

\section{Introduction}
A classical association scheme is an algebraic-combinatorial object with certain symmetry properties. These properties suffice to afford classical association schemes with desirable structural characteristics and are pliant enough to allow classical association schemes to be applicable to several areas of mathematics. For example, the adjacency algebra of a classical association scheme is semisimple and, when the adjacency matrices define a distance-regular graph, the structure constants of this algebra can be expressed in terms of certain families of orthogonal polynomials. \cite{bannai_algebraic_1984}


Mesner and Bhattacharya introduced the notion of association schemes on triples (or ASTs), a ternary analogue for classical association schemes \cite{mesner_association_1990}. An AST on a set $\Omega$ is a partition of the triple Cartesian product $\Omega \times \Omega \times \Omega$ subject to regularity requirements paralleling the symmetry conditions for classical association schemes. In ASTs, the resulting adjacency hypermatrices produce a ternary algebra under a ternary product that extends the usual matrix multiplication.

However, the structural properties of ASTs remain unclear, partly due to the ternary adjacency algebra not being associative nor commutative. As first steps in the investigation of ASTs, studies were conducted regarding analogues of identity and inverse elements \cite{mesner_ternary_1994}, enumerations of ASTs over the smallest number of vertices \cite{balmasmall}, possible sources of ASTs such as group actions, two-graphs, designs, and other ASTs \cite{mesner_association_1990,Zealand2021,balmasurvey}, as well as the intersection numbers of known families of ASTs \cite{mesner_association_1990,balmafamily,balmasurvey}.

In particular, the actions of two-transitive groups yield ASTs \cite{mesner_association_1990}. The orbits of these actions are closely related to the parameters of the AST, providing their sizes, third valencies, and intersection numbers \cite{mesner_association_1990,balmafamily}. In fact, \cite{mesner_association_1990} provides the sizes, third valencies, and intersection numbers of the ASTs obtained from the affine general linear group $AGL(1,n)$ where $n$ is a prime power. This was extended in \cite{balmafamily}, wherein these parameters were obtained for the ASTs obtained from subgroups of the affine semilinear group $A\Gamma L(k,n)$ of the form $AGL(k,n) \rtimes H$, where $k\geq 1$ and $H\leq Gal(GF(q)$. Further work was done in \cite{balmasurvey}, where these parameters were obtained from ASTs obtained from the affine special linear group $ASL(2,n)$. 

We extend this last result by determining the sizes, third valencies, and intersection numbers of ASTs obtained from subgroups of the affine special semilinear group $ASL(k,n)\rtimes Gal(GF(n))$ of the form $ASL(k,n) \rtimes H$, where $k\geq 2$, $n$ is a prime power, and $H$ is a subgroup of $Gal(GF(n))$. In particular, we show that the ASTs obtained from $ASL_H(k,n)$ are the same as the ASTs obtained from $AG L_H(k,n) = AGL(k,n)\rtimes H$ for $k\geq 3$.

\section{Preliminaries}
We define association schemes on triples, remark how ASTs arise from two-transitive groups, and review the actions of the affine special linear and affine special semilinear groups.

\subsection{Association schemes on triples}\label{section_ast}
We define association schemes on triples and state how the parameters of an AST obtained from a two-transitive group are related to the corresponding group action.

\begin{defi}\label{defn_AST}
\cite{mesner_association_1990,Zealand2021} Let $\Omega$ be a finite set with at least 3 elements. An association scheme on triples (AST) on $\Omega$ is a partition $X=\lr{R_i}_{i=0}^m$ of $\Omega \times \Omega \times \Omega$ with $m\geq 4$ such that the following hold.

\begin{enumerate}
    \item  For each $i\in \lr{0,\ldots,m}$, there exists an integer $n_i^{(3)}$ such that for each pair of distinct $x,y\in \Omega$, the number of $z\in \Omega$ with $(x,y,z)\in R_i$ is $n_i^{(3)}$.
    \item (Principal Regularity Condition.) For any $i,j,k,l \in \lr{0,\ldots,m}$, there exists a constant $p_{ijk}^l$ such that for any $(x,y,z)\in R_l$, the number of $w$ such that $(w,y,z)\in R_i$, $(x,w,z)\in R_j$, and $(x,y,w)\in R_k$ is $p_{ijk}^l$.
    \item For any $i\in \lr{0,\ldots, m}$ and any $\sigma \in S_3$, there exists a $j\in \lr{0,\ldots,m}$ such that \[R_j= \lr{(x_{\sigma(1)},x_{\sigma(2)},x_{\sigma(3)}):(x_1,x_2,x_3)\in R_i}.\]
    \item The first four relations are $R_0=\lr{(x,x,x): x\in \Omega}$, $R_1=\lr{(x,y,y):x,y\in \Omega,x\neq y}$, $R_2=\lr{(y,x,y):x,y\in \Omega, x\neq y}$, and $R_3=\lr{(y,y,x):x,y\in \Omega, x\neq y}$.
\end{enumerate}
\end{defi}

The integer $n_i^{(3)}$ is the third valency of $R_i$, and is the analogue of valency from classical association schemes. By Conditions 1 and 3 of Definition \ref{defn_AST} there are for each $i$ the constants $n_i^{(1)}=\vert \lr{z\in \Omega : (z,x,y)\in R_i} \vert $ and $n_i^{(2)}=\vert \lr{z\in \Omega : (x,z,y)\in R_i} \vert $ independent of any pair of distinct $x,y\in \Omega$. Similarly, $n_i^{(1)}$ is the first valency of $R_i$ and $n_i^{(2)}$ is the second valency of $R_i$. The trivial relations are $R_0,R_1,R_2$ and $R_3$ while the other relations are the nontrivial relations. Further, the numbers $p_{ijk}^l$ are called the intersection numbers.

ASTs arise naturally from the actions of two-transitive groups \cite{mesner_association_1990}, mirroring how Schurian classical association schemes are induced by the actions of transitive groups \cite{bannai_algebraic_1984}. Indeed, when a two-transitive group $G$ acts on a set $\Omega$, the orbits of the induced action on $\Omega \times \Omega \times \Omega$ is an AST \cite{mesner_association_1990}. Correspondences between the action and the parameters of the induced AST are summarized in the following remark.
\begin{rem}[\cite{mesner_association_1990,balmafamily}] \label{rem_parameter}
Let $G$ be a group acting two-transitively on a set $\Omega$ and let $X$ be the AST induced by this action. For any pair of distinct elements $a,b\in \Omega$, the orbits of the two-point stabilizer $G_{a,b}$ on $\Omega\setminus \{a,b\}$ are in bijection with the nontrivial relations of the AST. As a consequence of this bijection, the sizes of these orbits are also the third valencies.    
\end{rem}

\subsection{Affine special groups}
Given a prime power $n$ and $k\geq 1$, the affine special linear group $ASL(k,n)$ is the semidirect product $GF(n)\rtimes SL(k,n)$, where $SL(k,n)$ is the group of invertible linear transformations on the $k$-dimensional vector space $V$ over $GF(n)$ of determinant 1. Explicitly, the affine special linear group is the following group of maps from $V$ to itself.
\[ASL(k,n) = \left\{ x\mapsto Ax +b : A \in SL(k,n), b\in V \right\}.\] 

Similarly, the affine special semilinear group $ASL(k,n) \rtimes Gal(GF(n))$ is the semidirect product of the affine special linear group $ASL(k,n)$ with the Galois group $Gal(GF(n))$. Explicitly, the affine special semilinear group is the following group of maps from $V$ to itself.
\[ASL(k,n) \rtimes Gal(GF(n)) = \left\{ x\mapsto A\phi(x) +b : A\in SL(k,n),b\in GF(n), \phi \in Gal(GF(n)) \right\}.\]

\section{ASTs from subgroups of the affine special semilinear group}
In this section we generalize work done in \cite{balmasurvey} by obtaining the sizes, third valencies, and intersection numbers of ASTs obtained from the actions of subgroups of the affine special semilinear group of the form \[ASL_H(k,n) = ASL(k,n )\rtimes H,\] where $k\geq 2$, $n=p^\alpha$ a power of a prime number $p$, and $H$ a subgroup of $Gal(GF(n))$. We obtain the sizes and third valencies of these ASTs by obtaining a two-point stabilizer of $ASL_H(k,n)$ and then determining its orbits. Finally, we obtain the intersection numbers of these ASTs through explicit orbit computations.

For ease of discussion, we fix the following notations. Let $n=p^\alpha$ be a power of a prime $p$, $k\geq 2$, $V$ be the $k$-dimensional vector space over $GF(n)$, $H$ be a subgroup of $Gal(GF(n))$, and $X$ be the AST obtained from $ASL_H(k,n)$. For $a\in GF(n)$, let $\vec{a}=(a,0,\ldots,0)^T\in V$. Further, 
for $(u,v,w)\in V\times V\times V$, let $[(u,v,w)]\in X$ denote the orbit of $(u,v,w)$ under $ASL_H(k,n)$.

 We begin with the case where $k=2$. To determine the size and third valencies of $X$, we use Remark \ref{rem_parameter} exploit the relationships between these parameters and the orbits of a two-point stabilizer of $ASL_H(k,n)$.

\begin{thm} \label{thm_asl2}
Let $n=p^\alpha$ be a power of a prime $p$, $q=p^\omega$ with $\omega|\alpha$, $H= Gal_{GF(q)}(GF(n))$ and $X$ be the AST obtained from the action of $ASL_H(2,n)$ on the $2$-dimensional vector space $V$ over $GF(n)$.
The two-point stabilizer $ASL_H(2,n)_{\vec{0},\vec{1}}$ has the following orbits on $V\setminus \lr{\vec{0},\vec{1}}$.
\begin{enumerate}
    \item There are $-2 + \frac{\omega}{\alpha} \sum_{\beta=1}^{\frac{\alpha}{\omega}}q^{\gcd{(\frac{\alpha}{\omega},\beta)}}$ orbits of the form \[\left\{ \vec{\phi(a)}: \phi \in H   \right \},\;a\neq 0,1\] each of size $\deg_{GF(q)}(a)$.
    \item There are $-1 + \frac{\omega}{\alpha} \sum_{\beta=1}^{\frac{\alpha}{\omega}}q^{\gcd{(\frac{\alpha}{\omega},\beta)}}$ orbits of the form \[ \left\{ (c,\phi(\mathfrak{a}))^T: c\in GF(n), \phi \in H   \right \},\;\mathfrak{a}\neq 0\] each of size $n\deg_{GF(q)}(a)$.
\end{enumerate}

\end{thm}

\begin{proof}
The two-point stabilizer is \[ASL_H(2,n)_{\vec{0},\vec{1}} = \{(x,y)^T\mapsto \begin{bmatrix} 1 & c\\ 0 & 1\end{bmatrix}(\phi(x),\phi(y))^T:c \in GF(n), \phi \in H\}.\]

Direct computation shows that the orbits of $ASL_H(2,n)_{\vec{0},\vec{1}} $ have the following forms.    
\begin{enumerate}
    \item The first type of orbit has the form \[\{(\phi(a),0)^T:\phi \in H\}, \] which consists of those vectors whose second coordinate is 0 and whose first coordinate is a Galois conjugate of an element $a\in GF(n)$ with $a\neq0,1$. 
    \item The remaining orbits are of the form \[\{(x,\phi(\mathfrak{a}))^T: x\in GF(n),\phi \in H\}, \] which consists of those vectors whose second coordinate is a Galois conjugate of an element $\mathfrak{a}\in GF(n)$ with $\mathfrak{a}\neq0$. 
\end{enumerate}
The sizes of these orbits follow directly from the Fundamental Theorem of Galois Theory. The number of orbits of each type are then obtained through the Fundamental Theorem of Galois Theory and a straightforward application of Burnside's Orbit Counting Theorem to the action of $Gal(GF(n))$ on $GF(n)$. 
\end{proof}

As a consequence of Theorem \ref{thm_asl2}, we obtain the sizes and third valencies of the ASTs obtained from $ASL_H(2,n)$.

\begin{thm}

Let $n=p^\alpha$ be a power of a prime $p$, $q=p^\omega$ with $\omega|\alpha$, $H= Gal_{GF(q)}(GF(n))$ and $X$ be the AST obtained from the action of $ASL_H(2,n)$ on the $2$-dimensional vector space $V$ over $GF(n)$. Then $X$ has \(-3 + 2\left(\frac{\omega}{\alpha} \sum_{\beta=1}^{\frac{\alpha}{\omega}}q^{\gcd{(\frac{\alpha}{\omega},\beta)}}\right)\) nontrivial relations. 
There are $-2 + \frac{\omega}{\alpha} \sum_{\beta=1}^{\frac{\alpha}{\omega}}q^{\gcd{(\frac{\alpha}{\omega},\beta)}}$ nontrivial relations of the form \[R^a = \lr{[(\vec{0},\vec{1},\vec{a})]}, \; a\neq0,1,\] with corresponding third valency $\deg_{GF(q)}(a)$. The remaining $-1 + \frac{\omega}{\alpha} \sum_{\beta=1}^{\frac{\alpha}{\omega}}q^{\gcd{(\frac{\alpha}{\omega},\beta)}}$ nontrivial relations of $X$ are of the form  \[{}^\mathfrak{a}\!R=\lr{[(\vec{0},\vec{1},(0,\mathfrak{a})^T)]}, \; \mathfrak{a}\neq0,\] with corresponding third valency $n \deg_{GF(q)}(\mathfrak{a})$. 
\end{thm}

For notational convenience, let $A^a$ denote the adjacency hypermatrix corresponding to the relation $R^a$ whenever $a\neq0,1$. Similarly, let ${}^\mathfrak{a} \! A$ denote the adjacency hypermatrix corresponding to the relation ${}^\mathfrak{a} \!R$ whenever $\mathfrak{a} \neq 0$. Further, let $T$ be a transversal of the orbits of $H$ on $GF(n)\setminus \{0\}$. The intersection numbers of the subalgebra generated by the adjacency hypermatrices of the nontrivial relations of $X$ are given in the next theorem. 

\begin{thm}\label{thm_intnums}
Let $n=p^\alpha$ be a power of a prime $p$, $q=p^\omega$ with $\omega|\alpha$, $H= Gal_{GF(q)}(GF(n))$ and $X$ be the AST obtained from the action of $ASL_H(2,n)$. The following equations hold for any $a,b,c\neq0,1$ and $\mathfrak{a},\mathfrak{b},\mathfrak{c}\neq 0$. 
\begin{enumerate}
    \item $A^a A^b A^c = \sum_{\ell \in T\setminus\{1\}} p_\ell A^\ell$, where \[p_\ell = \left| \left\{  \phi(c) : \phi\in H \text{ and } \left(\exists\psi,\tau\in H \right)\left[ (1-\phi(c))\tau(a)+\phi(c)=\ell=\phi(c)\psi(b)\right]  \right\}  \right|.\]

    \item $A^a A^b \,{}^\mathfrak{c}\!A= A^a\, {}^\mathfrak{c}\!A\, A^b={}^\mathfrak{c}\!A \, A^a  A^b=0  $.

    \item ${}^\mathfrak{a}\!A \, {}^\mathfrak{b}\! A \,A^c = \sum_{\ell \in T} p_\ell\; {}^\mathfrak{\ell}\!A$, where \[p_\ell = \left| \left\{  \phi(c) : \phi\in H \text{ and } \left(\exists\psi,\tau\in H \right)\left[ \frac{\tau{(\mathfrak{a})}}{1-\phi(c)}=\ell=\frac{\psi(\mathfrak{b})}{\phi(c)}\right]  \right\}  \right|.\]

    \item ${}^\mathfrak{a}\!A \, A^c\, {}^\mathfrak{b}\! A  = \sum_{\ell \in T} p_\ell\; {}^\mathfrak{\ell}\!A$, where \[p_\ell = \left| \left\{  \psi(\mathfrak{b}) : \psi\in H \text{ and } \left(\exists\phi,\tau\in H \right)\left[ \psi(\mathfrak{b})\phi(c)=\ell=\tau(\mathfrak{a})+\psi(\mathfrak{b})\right]  \right\}  \right|.\]

    \item $A^c\,{}^\mathfrak{a}\!A \, {}^\mathfrak{b}\! A  = \sum_{\ell \in T} p_\ell\; {}^\mathfrak{\ell}\!A$, where \[p_\ell = \left| \left\{  \psi(\mathfrak{b}) : \psi\in H \text{ and } \left(\exists\phi,\tau\in H \right)\left[ \psi(\mathfrak{b})(1-\phi(c))=\ell=\frac{\tau(\mathfrak{a})(\phi(c)-1)}{\phi(c)}\right]  \right\}  \right|.\]

    \item ${}^\mathfrak{a}\!A \, {}^\mathfrak{b}\! A \, {}^\mathfrak{c}\!A = \sum_{\ell\in T\setminus\{1\}}p_\ell A^\ell + \sum_{{\jmath}\in T} p_{\jmath} \;{}^{\jmath}\! A$, where
    \begin{align*}
       p_\ell &= q\left|\left\{ \phi(\mathfrak{c}):(\exists \psi, \tau \in H)\left[  \frac{\tau(\mathfrak{a})+\phi(\mathfrak{c})}{\phi{(\mathfrak{c})}}=d=-\frac{\psi(\mathfrak{b})}{\phi(\mathfrak{c})}\right]  \right\}\right|, \\
              p_\jmath &= \left|\left\{ \psi(\mathfrak{b}):(\exists \phi,\tau \in H)\left[ \tau(\mathfrak{a}) + \psi(\mathfrak{b})+\phi(\mathfrak{c}) = \jmath \right]  \right\}\right|.
    \end{align*}
\end{enumerate}

\end{thm}

\begin{proof}
We prove only the third statement, as the other statements are shown similarly. With $R_i= {}^\mathfrak{a}\!R$, $R_j= {}^\mathfrak{b}\!R$, and $R_k=R^c$, we determine the $R_\ell$ such that the intersection number $p_{ijk}^\ell$ is nonzero. If $R_\ell= {}^d\!R$ for some $d\neq 0$, considering the viable $w$ as in the the Principal Regularity Condition from Definition \ref{defn_AST} necessitates that $\phi (c) \psi(b) =0$ for some $\phi,\psi \in H$, which is impossible. If $R_\ell = R^d$ for some $d\neq0,1$, the Principal Regularity Condition says that the number $p_{ijk}^\ell$ of viable $w$ is the number of vectors of the form $(\phi(c),0)^T$ with $\phi \in H$ such that there are $\psi$ and $\tau$ in $H$ that satisfy $\dfrac{\tau{(\mathfrak{a})}}{1-\phi(c)}=\ell=\dfrac{\psi(\mathfrak{b})}{\phi(c)}$  
\end{proof}

The succeeding theorem gives the intersection numbers $p_{ijk}^l$ of the ASTs obtained from $ASL_H(2,q)$ whenever exactly one of $R_i$, $R_j$, and $R_k$ is trivial. Here $I_1$, $I_2$, and $I_3$ denote the respective adjacency hypermatrices of the trivial relations $R_1$, $R_2$, and $R_3$ of $X$. The proof is omitted, being similar to that of the proof of Theorem \ref{thm_intnums}.

\begin{thm}
Let $n=p^\alpha$ be a power of a prime $p$, $q=p^\omega$ with $\omega|\alpha$, $H= Gal_{GF(q)}(GF(n))$ and $X$ be the AST obtained from the action of $ASL_H(2,q)$. The following equations hold for any $a,b\neq0,1$ and $\mathfrak{a},\mathfrak{b}\neq 0$. 
\begin{enumerate}
\item $I_1 A^a A^b = p I_1$, where  \[p^1 = \vert \lr{\psi(b) : \psi \in H \text{ and } (\exists \tau\in H)[\tau(a) \psi(b)=1]}\vert .\]
    
\item $A^a I_2 A^b = p^2 I_2$, where \[p^2 = \vert \lr{\psi(b) : \psi \in H \text{ and } (\exists \tau\in H)[\tau(a) \psi(b)=\tau(a)+ \psi(b)]}\vert .\]
    
\item$ A^a A^b I_3 = p^3 I_3$, where \[p^3 =  \vert \lr{\psi(b) : \psi \in H \text{ and } (\exists \tau\in H)[\tau(a)+ \psi(b)=1]}\vert . \]
    \item $I_1 A^a\, {}^{\mathfrak{a}}\!A = I_1 \, {}^{\mathfrak{a}}\!A\, A^a  = A^a I_2 \, {}^{\mathfrak{a}}\!A = \, {}^{\mathfrak{a}}\!A\, I_2 A^a =A^a \, {}^{\mathfrak{a}}\!A\, I_3 = \, {}^{\mathfrak{a}}\!A\, A^a I_3 =0$.
    
    \item $I_1 \, {}^{\mathfrak{a}}\!A\, {}^{\mathfrak{b}}\!A = p^*  I_1$,  $ {}^{\mathfrak{a}}\!A\, I_2 \, {}^{\mathfrak{b}}\!A = p^*  I_2$, $ {}^{\mathfrak{a}}\!A\, {}^{\mathfrak{b}}\!A\,I_3  = p^*  I_3$, where \[p^* = q\left| \left\{ \psi(\mathfrak{b}): \psi\in H \text{ and } (\exists \tau \in H)[\tau(\mathfrak{a})=-\psi(\mathfrak{b})]   \right\} \right|.\]

\end{enumerate}

\end{thm}

We next consider the AST obtained from $ASL_H(k,n)$ for $k\geq 3$, $n$ a prime power, and $H$ a subgroup of $Gal(GF(n))$. The following theorem tells us that the AST obtained from $ASL_H(k,n)$ is the same as the AST obtained from the subgroup $A G L_H(k,n) = AGL(k,n) \rtimes H$ of the affine semilinear group $A\Gamma L(k,n)$ whenever $k\geq 3$. In particular, the parameters of these ASTs have already been obtained in \cite{balmasurvey}. 

\begin{thm}
 Let $n=p^\alpha$ be a power of a prime $p$, $q=p^\omega$ with $\omega|\alpha$, and $H= Gal_{GF(q)}(GF(n))$. Then the AST obtained from the action of $ASL_H(k,n)$ is equal to the AST obtained from the action of $AG L_H(k,n)$.   
\end{thm}

\begin{proof}
Notice that if a group $G$ and a subgroup $K$ of $G$ both act two-transitively on a set, then the orbits of $G$ are unions of orbits of $K$. In particular, if $G$ and $K$ have the same number of orbits, then these orbits are the same. Thus, it suffices to show that the ASTs obtained from $ AG L_H(k,n)$ and $ASL_H(k,n)$ have the same size. By Remark \ref{rem_parameter}, it suffices to show that the two-point stabilizer $ASL_H(k,n)_{\vec{0},\vec{1}}$ has the same orbits as $AG L_H(k,n)_{\vec{0},\vec{1}}$ on $GF(n)\setminus \{\vec{0},\vec{1}\}$.    

Indeed, the two-point stabilizers above are given by \[ASL_H(k,n)_{\vec{0},\vec{1}} = \{v\mapsto A\phi(v):A\in SL(k,n), \phi \in H\},\] and
\[AG L_H(k,n)_{\vec{0},\vec{1}} = \{v\mapsto A\phi(v):A\in GL(k,n), \phi \in H\}.\]

Direct computation shows that the orbits of $ASL_H(k,n)_{\vec{0},\vec{1}} $ have the following forms. 
\begin{enumerate}
    \item One type of orbit has the form \[\{(\phi(a),0,\ldots,0)^T:\phi \in H\}, \] which consists of those vectors whose first coordinate is a Galois conjugate of an element $a\in GF(n)$ with $a\neq0,1$. The other coordinates are 0.
    \item The remaining orbit is \[(GF(n))^k\setminus Span(\vec{1}),\] consisting of the vectors linearly independent from $\vec{1}$. 
\end{enumerate}

These are also the orbits of $AGL_H(k,n)_{\vec{0},\vec{1}}$, completing the proof.
\end{proof}


\bibliographystyle{amsplain}
 \bibliography{asl}





\end{document}